\documentclass[12pt,reqno]{amsart}
\usepackage{fullpage} 

\usepackage{graphicx} 
\usepackage{amsmath}
\usepackage{amssymb}
\usepackage{enumitem}
\usepackage{amsthm}
\usepackage{xcolor}
\usepackage{tikz}

\newtheorem{theorem}{Theorem}
\newtheorem{conjecture}{Conjecture}[section]
\newtheorem{proposition}[conjecture]{Proposition}
\theoremstyle{definition}
\newtheorem{definition}{Definition}[section]
\newtheorem{remark}[definition]{Remark}

\newcommand{\C}{\mathbb{C}}
\newcommand{\Z}{\mathbb{Z}}

\newcommand{\R}{{(r)}}
\newcommand{\g}{\mathfrak{g}}
\newcommand{\h}{\mathfrak{h}}
\renewcommand{\epsilon}{\varepsilon}
\renewcommand{\H}{\mathfrak{H}}
\newcommand{\V}{\mathfrak{V}}

\DeclareMathOperator{\id}{id}

\DeclareMathOperator{\End}{End}
\DeclareMathOperator{\Aut}{Aut}
\DeclareMathOperator{\spa}{span}

\title{Indecomposable Hopf $*$-algebra representations with invariant inner product}

\author{Quinn T. Kolt}
\email[corresponding author]{quinn@math.ucsb.edu}
\author{Ziqian Zhao}
\email{ziqian\_zhao@ucsb.edu}

\address{Department of Mathematics, University of California, Santa Barbara, CA 93106, USA}

\begin{document}

\begin{abstract}
    We generalize a result of Araki (1985) on indecomposable group representations with invariant (necessarily indefinite) inner product and irreducible subrepresentation to  Hopf $*$-algebras. Moreover, we characterize invariant inner products on the projective indecomposable representations of small quantum groups $U_q sl(2)$ at odd roots of unity and on the indecomposable representations of generalized Taft algebras $H_{n,d}(q)$.
\end{abstract}
\maketitle
\tableofcontents

\section{Introduction}
    Classically, when studying quantum symmetry, one considers unitary operators which commute with the Hamiltonian $\hat H$ of a quantum system $(\H, \hat H)$. Given a Lie group $G$, one says the system is \textit{$G$-invariant} if there is a unitary representation $\pi:G\to\Aut(\H)$ on the state space $\H$ such that $ \pi(g)^\dagger\hat H\pi(g) = \hat H$ for all $g\in G$. Define a Hermitian form $\langle v,w\rangle:= \langle v| \hat H|w\rangle$ for $v,w\in\H$, where $\langle -|- \rangle$ is the Hilbert space inner product on $\H$. The commutation relation $\pi(g)^\dagger\hat H\pi(g) = \hat H$ is then easily seen to be equivalent to the property that $\langle\pi(g)v,\pi(g)w\rangle = \langle v, w\rangle$ for all $v,w\in\H$ and $g\in G$. Thus, the classical notion of quantum symmetry can be captured by $G$-invariance of a Hermitian form. 

    Given the recent mathematical progress in non-semisimple quantum field theory (see, for instance, \cite{10.1112/jtopol/jtu006, Creutzig2024QFT, kerler2001non}), one might expect a corresponding notion of non-semisimple quantum symmetry. We may ask that the indefinite inner product $\langle -,-\rangle$ be instead preserved by some non-semisimple representation. Such a representation cannot be unitary. Nevertheless, the famous Gupta-Bleuler formalism in quantum electrodynamics (QED) was developed in such a setting \cite{Bleuler1950,Gupta_1950}. The formalism constructs a positive definite inner product from the indefinite inner product via a suitable restriction and quotient. While the rigorous foundations of QED remain unclear, indecomposable representations with invariant indefinite inner products have seen considerable applications within this formalism (e.g., \cite{PhysRevD.57.6230,gazeau2000gupta,Gazeau:1984bs}). Araki (1985) develops an outstanding mathematical theory with such inner products, treating a Gupta-Bleuler triplet as a special case \cite{Araki85}. Our main result is a non-invertible generalization of \cite[Thm. 1]{Araki85}, which may be applicable to systems with more complex symmetries than QED.
    
    A sesquilinear form on a complex vector space $\H$ is a map $\langle-,-\rangle:\H\times \H\to \C$ which is conjugate linear in the first coordinate and linear in the second. We may interpret this as a linear map $\langle-,-\rangle:\bar \H\otimes \H\to \C$, where $\bar \H$ is the conjugate vector space of $\H$, whose scalar action is defined by $\lambda\cdot_{\bar \H} v := \bar\lambda\cdot_\H v$ for $v\in \bar \H=\H$. Thus, $\langle-,-\rangle\in (\bar \H\otimes \H)^\vee$. Following \cite{Bais_2003}, given a representation $\pi:A\to\End(\H)$ of a Hopf algebra $A=(A, \nabla, 1, \Delta, \epsilon, S)$, we say a vector $v\in \H$ is $A$-invariant if $\pi(h)v = \epsilon(h)v$ for all $h\in A$. Thus, for a proper notion of invariance of $\langle -,-\rangle$, there must be an $A$-action on $(\bar \H\otimes \H)^\vee$. For a general Hopf algebra, there is no obvious representation on $\bar \H$. Hopf $*$-algebras $A$ provide the necessary structure to define such a representation. We study (possibly non-semisimple) Hopf $*$-algebras and their associated notion of symmetry. This notion simultaneously generalizes (indefinite) inner products invariant under the action of groups and of real Lie algebras. It may be possible to further extend our results to inner products on objects in bar categories, as defined by \cite{BeggsMajid06}.

    In Section \ref{sec:araki}, we review Hopf $*$-algebras and some notions from \cite{Araki85}. We then generalize \cite[Thm. 1]{Araki85} on a form of certain non-semisimple representations with invariant inner product and an irreducible subrepresentation from the group case to the Hopf $*$-case. In Section \ref{sec:ex}, we provide nontrivial examples of our theory. We study inner products on indecomposable representations of small quantum groups $U_q sl(2)$ at odd roots $q$ of unity and generalized Taft algebras $H_{n,d}(q)$. We show that every projective indecomposable $U_q sl(2)$-module has a compatible non-degenerate inner product. For some choices of $n, d$, no indecomposable $H_{n,d}(q)$-modules have a compatible non-degenerate inner product. This is in particular true for Sweedler's 4-dimensional Hopf algebra. In general, only certain indecomposable $H_{n,d}(q)$-modules have a compatible non-degenerate inner product.
    
\section{Inner products invariant under the action of a Hopf $*$-algebra}\label{sec:araki}

\begin{definition}
    A \textit{Hopf $*$-algebra} is a tuple $A=(A, \nabla, 1, \Delta, \varepsilon, S, *)$, where $(A, \nabla, 1, \Delta, \varepsilon, S)$  is a Hopf algebra and $*:A\to A$ is an conjugate linear involution such that
    \begin{enumerate}[label=(\roman*)]
        \item $1^* = 1$,
        \item $(hk)^* = k^*h^*$ for all $h,k\in A$, 
        \item $\Delta(h^{*}) = (*\otimes *)\circ\Delta(h)$ for all $h\in A$,
        \item $S(S(h)^*)^* = h$ for all $h\in A$.
    \end{enumerate}
\end{definition}
Axioms (i) and (ii) alone define a unital $*$-algebra. Axioms (iii) and (iv) are not necessary for our main result Theorem \ref{thm:araki}. However, it follows from these axioms that $\epsilon(h^*) = \overline{\epsilon(h)}$ and the antipode $S$ is invertible with $S^{-1}=*\circ S\circ *$.

Throughout this section, let $A$ be a Hopf $*$-algebra (which need not be finite-dimensional). We will make use of sum-less Sweedler's notation, writing $\Delta(x)=x^{(1)}\otimes x^{(2)}$. Now, we recall and generalize some notions from \cite{Araki85} to the Hopf $*$-case. In this article, we are primarily interested indecomposable representations $\pi:A\to\End(\H)$ on a complex vector space $\H$ which may be written in the following form:
\begin{equation}
  \left.\begin{aligned}
    \pi_n\to\pi_{n-1}\to \cdots \to \pi_1\\
    \H=\H_n\supset \H_{n-1}\supset \cdots\supset \H_1
\end{aligned}\right\},\label{decomposition}
\end{equation}
where $\H_j$ is a $A$-invariant subspace of $\H_{j+1}$ without an $A$-invariant complement, i.e., there is no $A$-invariant subspace $\H_j' \subset \H_{j+1}$ such that $\H_j \oplus \H_j' = \H_{j+1}$. For $j=2,\dots, n$, $\pi_j:A\to\End(\H_j/\H_{j-1})$ is the representation on $\H_j/\H_{j-1}$ induced by $\pi$. Finally, $\pi_1:A\to\End(\H_1)$ is the representation on $\H_1$ induced by $\pi$.

\begin{definition}
    Let $\pi:A\to\End(\H)$ be a representation of a Hopf $*$-algebra $A$ on a complex vector space $\H$. Let $\bar\H$ be the conjugate vector space to $\H$. The \textit{conjugate representation} $\bar \pi:A\to\End(\bar\H)$ is defined by $\bar \pi(h) v := \pi(S(h)^*)v$ for $h\in A$ and $v\in \bar \H=\H$.
\end{definition}

\begin{definition}
    Let $\pi:A\to\End(\H)$ be a representation of a Hopf $*$-algebra $A$ on a complex vector space $\H$. Let $\langle-,-\rangle:\bar \H\otimes \H \to \C$ be a Hermitian form on $\H$. We say that $\langle-,-\rangle$ is \textit{$A$-invariant} if 
    $$((\bar\pi\otimes\pi)^\vee(h)\cdot \langle-,-\rangle)(\xi\otimes\eta) = \langle \pi(S^2(h^{(2)})^*)\xi,\pi(S(h^{(1)}))\eta\rangle=\varepsilon(h)\langle\xi,\eta\rangle$$ 
    for all $h\in A$ and $\xi,\eta \in \H$. We say $\langle -,-\rangle$ is an \textit{inner product} if it is non-degenerate. (We do not assume $\langle -,-\rangle$ is positive definite.) 
\end{definition}

\begin{remark}
    Let $G$ be a group. Then, the Hopf algebra $\C[G]$ has a natural $*$ generated by
    \begin{align*}
        g^{*} = g^{-1},
    \end{align*}
    for all $g\in G$. Let $\pi:G\to\Aut(\H)$ be a representation of $G$. With this $*$-structure, the notions of $G$-invariant (i.e., $\langle \pi(g) v, \pi(g)w\rangle = \langle v, w\rangle$) Hermitian forms and $\C[G]$-invariant Hermitian forms coincide.
    
    Let $\g$ be a complex Lie algebra. For any real Lie subalgebra $\h\subset \g$ for which $\g$ decomposes as $\g=\h\oplus i\h$, there is a natural $*$ on the Hopf algebra $U(\g)$, which is generated by
    \begin{align*}
         (x_r + ix_i)^* = -x_r + ix_i,
    \end{align*}
    for all $x_r, x_i\in \h$. In fact, $*$-structures are in bijection with such decompositions \cite[Sec. V.9]{kassel}. In particular, if $\h$ is a real Lie algebra and $\pi:\h\to\End(\H)$ is a representation, then there is a natural correspondence between $\h$-invariant (i.e., $\langle \pi(x)v, w\rangle = -\langle v, x\cdot w\rangle$) real symmetric forms and $U(\h\oplus i\h)$-invariant Hermitian forms (with the induced $*$-structure from $\h$).

    In both the group and Lie algebra cases, the $*$ and $S$ operations commute. However, this is not true in general. In fact, this is only true when $S^2 = \id_A$. Consequently, in general, there are many possible notions of $A$-invariant inner products which depend on the definition of $\bar \H$. 
    Regardless of this choice, Theorem \ref{thm:araki} still holds.
\end{remark}

\begin{proposition}[\cite{BeggsMajid06}]\label{prop:beggs}
    Let $\pi:A\to\End(\H)$ be a representation of a Hopf $*$-algebra $A$ on a complex vector space $\H$. Let $\langle-,-\rangle:\bar \H\otimes \H\to \C$ be a Hermitian form. Then, the following are equivalent:
    \begin{enumerate}[label=(\roman*)]
        \item The element $\langle-,-\rangle\in (\bar \H\otimes \H)^\vee$ is $A$-invariant, i.e., for any $h\in A$,
        $$(\bar\pi\otimes \pi)^\vee (h) \cdot \langle -,-\rangle = \langle \pi(S^2(h^{(2)})^*)\cdot -,\pi(S(h^{(1)}))\cdot -\rangle = \epsilon(h)\langle -, -\rangle;$$
        \item The map $\langle-,-\rangle:\bar \H\otimes \H\to \C_\epsilon$ is $A$-linear (where $\C_\epsilon$ is the trivial representation), i.e., for any $h\in A$, $\bar m\in \bar \H$ and $m\in \H$,
        $$\langle (\bar\pi\otimes \pi) (h)\cdot (\bar m\otimes m)\rangle = \langle \pi(S(h^{(1)})^*)\bar m,\pi(h^{(2)}) m\rangle = \epsilon(h)\langle \bar m, m\rangle;$$
        \item $*$ agrees with the adjoint operation of $\langle-,-\rangle$, i.e., for any $h\in A$, $\bar m\in \bar \H$ and $m\in \H$, 
        $$\langle \pi(h^*)\bar m, m\rangle = \langle \bar m, \pi(h)m\rangle.$$
    \end{enumerate}
\end{proposition}

Looking at Proposition \ref{prop:beggs}(iii), one might wonder if the assumption of having a Hopf algebra structure is necessary at all. In the constructions discussed here, it is indeed unnecessary. Everything may be reformulated purely in terms of $*$-algebras. However, there are many other ways one might generalize our results. For example, $\bar\pi' = \pi\circ *\circ S^{n}$ for odd $n$ defines another notion of conjugate representation, which induces a distinct notion of invariant inner product. Theorem \ref{thm:araki} still applies in this case with minimal changes to its proof. It seems natural to also look for a further extension to bar categories, introduced in \cite{BeggsMajid06}. We focus on Hopf $*$-algebras with this notion of conjugate representation for simplicity and to align with our desired application to non-invertible symmetry.

\begin{definition}\label{def:conjugation}
    Let $A$ be a Hopf $*$-algebra. Representations $\rho_1:A\to \End(\V_1)$ and $\rho_2:A\to\End(\V_2)$ are \textit{conjugate} if there is a sesquilinear form $\langle-,-\rangle:\bar\V_2\otimes\V_1\to\C$, such that
    \begin{enumerate}
        \item $\V_1$ and $\V_2$ separate each other, i.e., $\langle\xi,\eta\rangle = 0$ for all $\eta\in \V_1$ implies $\xi=0$ and $\langle\xi,\eta\rangle = 0$ for all $\xi\in \V_2$ implies $\eta=0$;
        \item the form is $A$-invariant, i.e., the following holds for all $\eta \in \V_1, \xi \in \V_2$, and $h\in A$:
    \begin{equation}
         \langle \rho_2(S^2(h^{(2)})^{*})\xi,\rho_1(S(h^{(1)}))\eta\rangle = \varepsilon(h)\langle \xi,\eta\rangle.
    \end{equation}
    \end{enumerate}
\end{definition}

\begin{definition}
    Let $\pi:A\to\End(\H)$ be a representation of a Hopf $*$-algebra $A$ on a complex vector space $\H$ with an $A$-invariant Hermitian form $\langle-,-\rangle:\bar\H\otimes\H\to\C$. Let $\H_1 \subset \H$ be an $\pi(A)$-invariant subspace. The \textit{polar} $\H_1^\perp$ of $\H_1$ in $\H$ is
    $$\H_1^\perp = \{\xi \in \H | \langle\xi,\eta\rangle = 0\text{ for all }\eta \in \H_1\}.$$
    $\H_1$ is said to be \textit{closed} if $\H_1 = (\H_1^\perp)^\perp$. A representation $\pi$ of $A$ on $\H$ is \textit{topologically indecomposable} if there are no non-zero closed $A$-invariant subspaces $\H_1$ and $\H_2$ of $\H$ such that $((\H_1+\H_2)^\perp)^\perp = \H$ and $\H_1 \cap \H_2 = 0$.
\end{definition}
\begin{definition}
    Let $\pi:A\to\End(\H)$ be a representation of a Hopf $*$-algebra $A$ on a complex vector space $\H$. Let $\H_1$ be a closed $\pi(A)$-invariant subspace of $\H$. A $\pi(A)$-invariant subspace $\H_2\subset \H$ is a \textit{topological complement} of $\H_1$ if $((\H_1+\H_2)^\perp)^\perp = \H$, $\H_1\cap \H_2 = 0$.
\end{definition}
\begin{theorem}\label{thm:araki}
    Let $\pi:A\to\End(\H)$ be a representation of a Hopf $*$-algebra $A$ with an $A$-invariant inner product. Let $\H_1$ be an $\pi(A)$-invariant closed subspace of $\H$ such that the restriction $\pi_1$ of $\pi$ to $\H_1$ is irreducible and there are no $A$-invariant, closed topological complements of $\H_1$ in $\H$. Then,
    \begin{enumerate}
        \item $\H_1$ is a null space for the Hermitian form ($\xi,\eta \in \H_1$ implies $\langle \xi,\eta\rangle = 0$),
        \item $\pi$ is of the form \eqref{decomposition} with $n = 2$ or $3$,
        \item $\pi_n$ on $\H_n/\H_{n-1}$ is conjugate to $\pi_1$ on $\H_1$,
        \item $\H_{2}/\H_1$ has an $A$-invariant inner product.
    \end{enumerate}
\end{theorem}
\begin{proof}

The proof is analogous to \cite[Thm. 1]{Araki85} with the exception of the notion of a conjugate representation of a Hopf $*$-algebra (Definition \ref{def:conjugation}).
Aligning with the notation of \cite[Thm. 1]{Araki85}, let $\H_2 = \H_1^\perp$ and $\H_1^\# \equiv \H/\H_2.$ Let $\xi^\# = \xi+\H_2 \in \H/\H_2$, and define $\pi^\#:A\to \End(\H/\H_2)$ by $\pi^\#(h)(\xi^\#) = (\pi(h)(\xi))^\#$.  Further, define $\langle \xi^\#,\eta\rangle := \langle\xi,\eta\rangle$ for $\xi\in \H, \eta \in \H_1$. This is well-defined as $\H_1$ is $\pi(A)$-invariant. Note that, if $\langle \xi^\#,\eta\rangle = 0$ for all $\xi^\#\in \H_1^\#$, then $\eta = 0$; on the other hand, if $\langle \xi^\#,\eta\rangle = 0$ for all $\eta\in \H_1$, then $\xi \in \H_1^\perp$ and $\xi^\# = 0$. Thus, $\H_1^\#$ and $\H_1$ separate each other. 

By $A$-invariance of the inner product, we have, for any $h\in A, \xi\in \H, \eta\in \H_1$
\begin{align*}
    \langle \pi_1^\#(S^2(h^{(2)})^{*})\xi^\#,\pi_1(S(h^{(1)}))\eta\rangle= \langle \pi_1(S^2(h^{(2)})^*)\xi,\pi_1(S(h^{(1)}))\eta\rangle=\varepsilon(h)\langle \xi,\eta\rangle =\varepsilon(h)\langle \xi^\#,\eta\rangle.
\end{align*}
Therefore, $\pi_1^\#$ on $\H_1^\#$ is conjugate to $\pi_1$ on $\H_1$.

If $\H_2 = \H_1^\perp=\H_1$, we have the form (1) with $n=2$. If $\H_2 = \H_1^\perp \neq \H_1$, define $\pi_2: A\to \End(\H_2/\H_1)$. We then have the form with $n=3$. In either case, $\pi_n = \pi_1^\#$ is conjugate to $\pi_1$.

\end{proof}

We shall call the form \eqref{decomposition} constructed in the proof of Theorem \ref{thm:araki} the \textit{Araki form}.

\section{Examples}\label{sec:ex}
In this section, we discuss examples of Hopf $*$-invariant inner products on indecomposable modules with a simple submodule. Such modules automatically satisfy the conditions of Theorem \ref{thm:araki}. We construct all possible invariant inner products in a few cases. Moreover, we investigate at how the indecomposable modules decompose according to Theorem \ref{thm:araki}.

Small quantum groups $U_q sl(2)$ at a root of unity have a compatible $*$-operation. Invariant Hermitian forms on the simple modules over the $q$-deformed enveloping algebra of $sl(2)$ were studied in \cite[Chap. VI]{kassel}. This classification descends to $U_q sl(2)$. More generally, it is known that $U_q \g$-modules always have $U_q \g$-invariant inner products for any simple Lie group $\g$. They can be chosen to be compatible with the modular structure of the representation category \cite{kirillov1996inner}. We place our focus on projective indecomposable $U_q sl(2)$-modules in the case where $q$ is an odd root of unity. In this case, we show that the $U_q sl(2)$-invariant Hermitian forms are always 2-dimensional as a real vector space.

Generalized Taft Hopf algebras $H_{n,d}(q)$, introduced in \cite{1494155620040901}, also have a compatible $*$-operation. Only some choices of $n,d$ give rise to non-degenerate inner products. We exhibit all possible $H_{n,d}(q)$-invariant Hermitian forms on all indecomposable $H_{n,d}(q)$-modules. Depending on the module, the space of $H_{n,d}(q)$-invariant Hermitian forms is either zero- or one-dimensional. We also give a classification of which $n, d$ give rise to non-degenerate inner products.

\subsection{Small quantum groups $U_qsl(2)$}\label{sec:uqsl2}
Let $l\geq 3$ be odd and $q$ be a primitive $l^{th}$ root of unity. As an algebra, the small quantum group $U_qsl(2)$ is generated by $K,E,$ and $F$ with relations:
\begin{align*}
    E^l&=0,&F^l&=0,&K^l&=1,\\
    KE&=q^2 EK,&KF &= q^{-2}FK, & [E,F] &= \frac{K - K^{-1}}{q - q^{-1}}.
\end{align*}
The associated PBW basis of $U_qsl(2)$ is $\{E^mF^nK^k | 0\leq m,n,k\leq l\}$. The Hopf algebra structure on $U_qsl(2)$ is generated by
\begin{align*}
    \Delta(E) &= 1\otimes E + E\otimes K, & \Delta(F) &= K^{-1}\otimes F + F\otimes 1, & \Delta(K) &= K\otimes K,\\
    \epsilon(E) &= 0,&\epsilon(F)&=0,&\epsilon(K)&=1,\\
    S(E) &= -EK^{-1}, & S(F) &= -KF, & S(K) &= K^{-1}.
\end{align*}
There is a unique $*$ which makes $U_qsl(2)$ a Hopf $*$-algebra. This $*:U_qsl(2)\to U_qsl(2)$ fixes the generators:
$$E^* = E, \quad F^* = F,\quad K^* = K.$$

There are $l-1$ projective indecomposable $U_q sl(2)$-modules up to isomorphism, only one of which is simple. All of these have dimension $2l$; for each $1\leq r \leq l-1$, the projective indecomposable $U_q sl(2)$-module $P_r$ has basis $\{x_k^{(r)},y_k^{(r)}\}_{k=0}^{l-r-1} \cup \{a_n^{(r)},b_n^{(r)}\}_{n=0}^{r-1}$. The actions of the generators are as follows:
\begin{align*}
    E\cdot x_k^{(r)} &= [k][l-r-k]x_{k-1}^{(r)}, & F\cdot x_k^{(r)} &= x_{k+1}^{(r)}, & K\cdot x_k^{(r)} &= q^{l-r-1-2k}x_k^{(r)},\\
    E\cdot x_{0}^{(r)} &= 0, & F\cdot x_{l-r-1}^{(r)} &= a_{0}^{(r)},\\
    E\cdot y_k^{(r)} &= [k][l-r-k]y_{k-1}^{(r)}, & F\cdot y_k^{(r)} &= y_{k+1}^{(r)}, & K\cdot y_k^{(r)} &= q^{l-r-1-2k}y_k^{(r)},\\
    E\cdot y_k^{(r)} &= a_{r-1}^{(r)}, & F\cdot y_{l-r-1}^{(r)} &= 0,\\
    E\cdot a_n^{(r)} &= [n][r-n]a_{n-1}^{(r)}, & F\cdot a_n^{(r)} &= a_{n+1}^{(r)}, & K\cdot a_n^{(r)} &= q^{r-1-2n}a_n^{(r)},\\
    E\cdot a_{0}^{(r)} &= 0, & F\cdot a_{r-1}^{(r)} &= 0,\\
    E\cdot b_n^{(r)} &= [n][r-n]b_{n-1}^{(r)} + a_{n-1}^{(r)}, & F\cdot b_n^{(r)} &= b_{n+1}^{(r)}, & K\cdot b_n^{(r)} &= q^{r-1-2n}b_n^{(r)},\\
    E\cdot b_{0}^{(r)} &= x_{l-r-1}^{(r)}, & F\cdot b_{r-1}^{(r)} &= y_{0}^{(r)}.
\end{align*}
Here, in the $E$-actions, $k$ ranges from 1 to $l-r-1$ and $n$ ranges from $1$ to $r-1$. In the $F$-actions, $k$ ranges from 0 to $l-r-2$ and $n$ ranges from 0 to $r-2$. In the $K$-actions, $k$ ranges from 0 to $l-r-1$ and $n$ ranges from $0$ to $r-1$.

The $r$-dimensional $U_qsl(2)$-submodule $V_r$ of $P_r$ generated by $\{a_n^{(r)}\}_{n=0}^{r-1}$ is simple. All simple $U_qsl(2)$-modules are of this form. The $(2l-r)$-dimensional $U_qsl(2)$-submodule $W_r\subset P_r$ generated by $\{a_n^{(r)}\}_{n=0}^{r-1}\cup \{x_k^{(r)}, y_k^{(r)}\}_{k=0}^{l-r-1}$ is indecomposable but neither projective nor simple.

\begin{theorem}\label{Thm2}
    Every projective indecomposable $U_qsl(2)$-module $P_r, 1\leq r\leq l-1$ has compatible $U_qsl(2)$-invariant inner products. In particular, a Hermitian form $\langle-,-\rangle :\bar P_r\otimes P_r\to \C$ is $U_qsl(2)$-invariant if and only if there are $\alpha,\beta\in\mathbb{R}$ such that
    \begin{enumerate}[label=(\roman*)]
        \item $\langle a_n^\R,b_m^\R\rangle = \langle b_m^\R,a_n^\R\rangle = \alpha$ whenever $n+m = r-1$\label{relation:1},
        \item $\langle x_k^\R,y_j^\R\rangle = \langle y_j^\R,x_k^\R\rangle = \alpha$ whenever $k+j = l-r-1$\label{relation:2},
        \item $\langle b_n^\R,b_m^\R\rangle = \beta$ whenever $n+m = r-1$\label{relation:3},
        \item $\langle \eta,\xi\rangle = 0$ for any other $\eta,\xi\in\{x_k^\R,y_k^\R\}_{k=0}^{l-r-1} \cup \{a_n^\R, b_n^\R\}_{n=0}^{r-1}$\label{relation:4}.
\end{enumerate}
\end{theorem}
\begin{proof}

    Observe that, for every $k=0,\dots, l-r-1$, the action of $K$ on $x_k^\R$ and on $y_k^\R$ behave identically. Moreover, for $n=0,\dots, r-1$, the action of $K$ on $a_n^\R$ and $b_n^\R$ also behave identically. For conciseness, we let $\xi_k\in \{x_k^\R,y_k^\R\}$ and $\mu_n\in\{a_n^\R, b_n^\R\}$, where $k=0,\dots, l-r-1$ and $n=0,\dots, r-1$. By $U_qsl(2)$-invariance, the $K$-action induces the following relations: 
    \begin{align}
        q^{-2(l-r-1-k-j)}\langle \xi_k,\xi_j\rangle &= \langle K\xi_k, K^{-1}\xi_j \rangle  = \epsilon(K)\langle \xi_k,\xi_j\rangle=\langle \xi_k,\xi_j\rangle\label{eq:3}\\
        q^{-l+2(k+n+1)}\langle \xi_k,\mu_n\rangle &= \langle K\xi_k, K^{-1}\mu_n \rangle  = \epsilon(K)\langle \xi_k,\mu_n\rangle=\langle \xi_k,\mu_n\rangle\label{eq:4}\\
        q^{-2(r-1-n-m)}\langle \mu_n,\mu_m\rangle &= \langle K\mu_m, K^{-1}\mu_n \rangle  = \epsilon(K)\langle \mu_n,\mu_m\rangle=\langle \mu_n,\mu_m\rangle\label{eq:5}
    \end{align}
    
    Assuming $\langle \xi_k,\xi_j\rangle\neq 0$, Equation \eqref{eq:3} implies that $l$ divides $2(l-r-1-k-j)$. Since $l$ is an odd number, $l$ must divide $l-r-1-k-j$. Therefore, we have $r+1+k+j = \lambda l$ for some integer $\lambda$. This implies $j+k = \lambda l -r-1 \equiv l-r-1\pmod{l}$. The range $0\leq j+k\leq 2l-2r-2$ of the indices $j$ and $k$ for $\xi_j$ and $\xi_k$ simplifies the relation to $j+k = l-r-1$. A similar argument with Equation \eqref{eq:5} implies $n+m = r-1$ when $\langle \mu_n, \mu_m\rangle\neq 0$.
    
    Equation \eqref{eq:4} implies that $l$ divides $l-2(k+n+1)$. Since $l$ is odd, $l$ divides $k+n+1$. However, note that the indices $k$,$n$ are constrained by $1\leq k+n+1 \leq l-1$. This implies $l$ cannot divide $k+n+1$, and thus $\langle \xi_k,\mu_n\rangle = 0$ for any $k,n$.
    
    By $U_qsl(2)$-invariance, the $F$-action induces the following relations for $0\leq j,k\leq l-r-2$ and $0\leq n,m\leq r-2$:
    \begin{multline}
        \langle K^{-1}FKx_k^\R, Kx_j^\R\rangle + \langle  x_k^\R, -KFx_j^\R\rangle=
        q^{l-r-j-3}\langle x_{k+1}^\R,x_{j}^\R\rangle - q^{l-r-j-3}\langle x_{k}^\R,x_{j+1}^\R\rangle=0,\label{eq:6}
    \end{multline}
    It follows from Equation \eqref{eq:6} and induction that $\langle x_k^\R,x_{j}^\R\rangle = \langle x_{k'}^\R,x_{j'}^\R\rangle$ whenever $k+j=k'+j'$. One can obtain similar relations for $\langle y_k^\R,y_{j}^\R\rangle$ and $\langle x_k^\R,y_{j}^\R\rangle$, as well as pairs in $\{a_n^\R, b_n^\R\}_{n=0}^{r-1}$.
    Combined with the results from the $K$-action, the $E$-action on the non-extreme indices induce the same constraints. Finally, consider $E$- and $F$-actions on the following extreme indices and their implications: 
    \begin{align}
    E\cdot\langle x_0^\R,b_0^\R\rangle = 0 &\Rightarrow \langle x_0^\R,x_{l-r-1}^\R\rangle=0,\label{eq:8}\\
    F\cdot\langle y_{l-r-1}^\R,b_0^\R\rangle = 0 &\Rightarrow \langle y_{l-r-1}^\R,y_0^\R\rangle = 0,\label{eq:9}\\
    F\cdot\langle x_{l-r-1}^\R,a_{r-1}^\R\rangle = 0&\Rightarrow\langle a_0^\R,a_{r-1}^\R\rangle = 0,\label{eq:10}\\
    E\cdot \langle y_0^\R, b_{r-1}^\R\rangle = 0&\Rightarrow\langle a_0^\R,b_{r-1}^\R\rangle = \langle y_0^\R, x_{l-r-1}^\R\rangle,\label{eq:11}\\
    F\cdot \langle x_{l-r-1}^\R,b_{r-1}^\R\rangle= 0&\Rightarrow \langle a_0^\R,b_{r-1}^\R\rangle = \langle x_{l-r-1}^\R,y_0^\R\rangle\label{eq:12}.
    \end{align}
    We now collect all of our information together. Notice that the additional constraints do not affect products of the form $\langle b_n, b_m\rangle$, so we have \ref{relation:3}. Equations \eqref{eq:8} and \eqref{eq:9} imply that $\langle x_k^\R, x_j^\R\rangle = \langle y_k^\R, y_j^\R\rangle = 0$ for all $k,j$. Equation \eqref{eq:10} implies that $\langle a_n^\R,a_m^\R\rangle = 0$ for all $n,m$. Combined with Equation \eqref{eq:11} and Equation \eqref{eq:12}, we have \ref{relation:1}, \ref{relation:2}, and \ref{relation:4}. 
\end{proof}

\begin{proposition}
    For all $U_qsl(2)$-invariant inner products, the Araki form of $P_r\supset V_r$ is $P_r\supset W_r\supset V_r$. 
\end{proposition}
\begin{proof}
    Theorem \ref{Thm2} suggests that $V_r^\perp$ is generated by $\{x_k^\R,y_k^\R\}_{k=0}^{l-r-1}\cup\{a_n^\R\}_{n=0}^{r-1} \cong W_r$. Since $W_r$ does not coincide with $V_r$, we have an Araki form of length $3$:
    $$P_r \supset W_r \supset V_r.$$
\end{proof}

\begin{proposition}
    $P_r/W_r \cong V_r$ and $W_r/V_r\cong V_{l-r}\oplus V_{l-r}$ (where the two summands are orthogonal).
\end{proposition}
\begin{proof}
    The map $\phi: V_{l-r}\oplus V_{l-r} \to W_r/V_r$ given by $(v_n,0)\mapsto (x_n+y_n)+V_r$, and $ (0,v_n)\mapsto (x_n-y_n)+V_r$ is an an $U_qsl(2)$-module isomorphism. By Theorem \ref{Thm2}, the summands are indeed orthogonal.
\end{proof}

\subsection{Generalized Taft Hopf algebras $H_{n,d}(q)$}\label{sec:taft}
Fix $n,d \geq 2$, such that $d|n$. Let $q$ be a primitive $d^{th}$ root of unity. Throughout the section, let $m = n/d$. $H_{n,d}(q)$, as an algebra, is generated by two elements $g$ and $h$ subject to the following relations:
$$g^n = 1, \quad h^d=0, \quad hg=qgh.$$
The associated PBW basis of $H_{n,d}(q)$ is $\{g^ih^j|0\leq i\leq n-1, 0\leq j\leq d-1\}$. The Hopf algebra structure on $H_{n,d}(q)$ given by
\begin{align*}
    \Delta(g) &= g\otimes g, & \Delta(h) &= 1\otimes h+h\otimes g,\\
    \epsilon(g) &= 1, & \epsilon(h) &=0,\\
    S(g) &= g^{-1}, & S(h) &= -q^{-1}g^{-1}h.
\end{align*}
There is a unique $*$ which makes $H_{n,d}(q)$ a Hopf $*$-algebra. Similarly to $U_q sl(2)$, this $*: H_{n,d}(q) \to H_{n,d}(q)$ fixes the generators:
$$g^* = g,\quad h^* = h.$$

Let $\omega$ be a primitive $n^{th}$ root of unity so that $\omega^m=q$. Following the notation of \cite{li2012greenringsgeneralizedtaft}, for $1\leq l\leq d$ and $i\in \mathbb{Z}_n$, $M(l,i)$ is an $H_{n,d}(q)$-module with basis $\{v_0, v_1, \cdots, v_{l-1}\}$ and the following actions:
\begin{align*}
    g\cdot v_j &= \omega^iq^{-j}v_j, & h\cdot v_j &= v_{j+1}\\
    & & h\cdot v_{l-1} &= 0.
\end{align*}
Here, in the $g$-actions, $j$ ranges from $0$ to $l-1$. In the $h$-actions, $j$ ranges from $0$ to $l-2$.

For any $1\leq l\leq d, i\in\mathbb{Z}_n$, $M(l,i)$ is an indecomposable $H_{n,d}(q)$-module. In fact, the set $\{M(l,i) | i\in \mathbb{Z}_n, 1\leq l \leq d\}$ forms a complete list of non-isomorphic indecomposable $H_{n,d}(q)$-modules \cite{li2012greenringsgeneralizedtaft}. Additionally, $M(l,i)$ is projective if and only if $l=d$, and $M(l,i)$ is simple if and only if $l=1$.

\begin{theorem}\label{thm:3}
    Let $l\geq 2$ and $i\in\Z_n$. The $H_{n,d}(q)$-module $M(l, i)$ has compatible $H_{n,d}(q)$-invariant inner products if and only if $2i\equiv m(l-1)\pmod{n}$. In particular, a Hermitian form $\langle -,- \rangle: \overline{M(l,i)}\otimes M(l,i) \to \mathbb{C}$ is $H_{n,d}(q)$-invariant if and only if there is $\alpha\in \mathbb{R}$, such that
    \begin{enumerate}[label=(\roman*)]
        \item $\langle v_j, v_k\rangle = \alpha$ if $m(j+k)\equiv 2i \pmod{n}$ and $j+k<l$,
        \item $\langle v_j,v_k\rangle = 0$ otherwise.
    \end{enumerate}
\end{theorem}
Note that the conditions $m(j+k)\equiv 2i\pmod{n}$ and $j+k<l$ can be satisfied by at most one choice of $j+k$. It follows that the space of invariant Hermitian forms on $M(l, i)$ is either zero- or one-dimensional.

\begin{proof}
    Assuming $H_{n,d}(q)$-invariance, the action of $g$ induces
    \begin{equation}
        \omega^{-2i}q^{j+k}\langle v_j,v_k\rangle = \langle v_j,v_k\rangle \label{eq:13}
    \end{equation}
    for $0\leq j,k\leq l-1$. The action of $h$ induces
    \begin{equation}
        \langle v_j,v_{k}\rangle = \omega^{-2i}q^{j+k}\langle v_{j-1},v_{k+1}\rangle\label{eq:14}
    \end{equation}
    for $1\leq j \leq l-1,0\leq k\leq l-2$. 

    Equation \eqref{eq:14} indicates that $\langle v_j,v_k\rangle = 0$ if $j+k \geq l$, so we only focus on $j+k<l$ in the following part. Assume $\langle v_j,v_k\rangle \neq 0$, Equation \eqref{eq:13} implies that $\omega^{-2i}q^{j+k} = 1$, namely $m(j+k)\equiv 2i\pmod{n}$. Notice that for each chosen $i$, at most one specific $j+k$ satisfies the relation. Suppose a specific $j+k$ satisfies this relation, Equation \eqref{eq:14} implies that $\langle v_j,v_{k+1}\rangle = \langle v_{j+1},v_k\rangle$. Therefore, we can conclude that $\langle v_j,v_k\rangle\in\mathbb{R}$ depends only on the value of $j+k$ and can be nonzero only if $m(j+k)\equiv 2i\pmod{n}$ and $j+k<l$.

    We claim that the Hermitian form is non-degenerate if and only if $j+k = l-1$.
    Suppose by contrapositive, $j+k<l-1$ and consider $\langle v_{l-1}, -\rangle$. There does not exist a $k$ in range such that $l-1+k < l-1$. Therefore, $\langle v_{l-1}, v_k\rangle = 0$ for all $k$. Thus, the $H_{n,d}(q)$-invariant Hermitian form is degenerate.  
\end{proof}

    Note that $\spa \{v_{l-1}\} \subset M(l,i)$ is isomorphic to $M(1,i-m(l-1))$. We can therefore consider the Araki form of $M(l,i)$ with respect to the irreducible submodule $M(1,i-m(l-1))$.

\begin{proposition}
    For any $H_{n,d}(q)$-invariant inner product, the Araki form of $M(l,i)\supset M(1,i-m(l-1))$ is 
    \begin{enumerate}
        \item  $M(l,i)\supset M(1,i-m(l-1))$ if $l=2$,
        \item  $M(l,i)\supset M(l-1,i-m)\supset M(1,i-m(l-1))$ if $l>2$,
    \end{enumerate}
    where we consider the second parameter as an element in $\mathbb{Z}_n$.
\end{proposition}
\begin{proof}
    Theorem \ref{thm:3} implies $M(1,i+m(1-l))^\perp = \{v_{l-1}\}^\perp = \spa \{v_1, v_2,\dots, v_{l-1}\}\subset M(l,1)$. If $l=2$, $M(1,i-m(l-1))^\perp = \spa \{v_{l-1}\}$ coincides with $M(1,i-m(l-1))$, then we have an Araki form of length 2:
    $$M(l,i)\supset M(1,i-m(l-1)).$$
    If $l>2$, $M(1,i-m(l-1))^\perp = \spa \{v_1, v_2,\dots, v_{l-1}\}$ does not coincide with $M(1,i-m(l-1))$. Notice that $\spa \{v_1, v_2,\dots, v_{l-1}\}\cong M(l-1, i-m)$,  we have an Araki form of length 3:
    $$M(l,i)\supset M(l-1, i-m)\supset M(1,i-m(l-1)).$$
\end{proof}

\begin{proposition}
    If $l>2$, then $M(l,i)/M(l-1,i-m) \cong M(1,i-m(l-1))$ and $M(l-1,i-m)/M(1,i-m(l-1))\cong M(l-2,i-m)$, where the second parameter is considered as an element in $\mathbb{Z}_n$.
\end{proposition}

\section*{Acknowledgements}
The authors thank Q. Zhang and Z. Wang for helpful comments and guidance. Q.T.K. is supported by the National Science Foundation Graduate Research Fellowship Program under Grant No. 2139319. Any opinions, findings, and conclusions or recommendations expressed in this material are those of the authors and do not necessarily reflect the views of the National Science Foundation. 
\bibliographystyle{abbrv}
\bibliography{zbib}

\begin{thebibliography}{10}

\bibitem{Araki85}
H.~Araki.
\newblock {Indecomposable representations with invariant inner product. A theory of the Gupta-Bleuler triplet}.
\newblock {\em Communications in Mathematical Physics}, 97(1-2):149--159, 1985.

\bibitem{Bais_2003}
A.~F. Bais, B.~J. Schroers, and J.~K. Slingerland.
\newblock Hopf symmetry breaking and confinement in (2+1)-dimensional gauge theory.
\newblock {\em Journal of High Energy Physics}, 2003(05):068, 2003.

\bibitem{BeggsMajid06}
E.~J. Beggs and S.~Majid.
\newblock Bar categories and star operations.
\newblock {\em Algebras and Representation Theory}, 12(2):103--152, 2009.

\bibitem{Bleuler1950}
K.~Bleuler.
\newblock {Eine neue Methode zur Behandlung der longitudinalen und skalaren Photonen}.
\newblock {\em Helvetica Physica Acta}, 23(V):567, 1950.

\bibitem{10.1112/jtopol/jtu006}
F.~Costantino, N.~Geer, and B.~Patureau-Mirand.
\newblock {Quantum invariants of 3-manifolds via link surgery presentations and non-semi-simple categories}.
\newblock {\em Journal of Topology}, 7(4):1005--1053, 2014.

\bibitem{Creutzig2024QFT}
T.~Creutzig, T.~Dimofte, N.~Garner, and N.~Geer.
\newblock {A QFT for non-semisimple TQFT}.
\newblock {\em Advances in Theoretical and Mathematical Physics}, 28(1):161--405, 2024.

\bibitem{PhysRevD.57.6230}
S.~De~Bi\`evre and J.~Renaud.
\newblock {Massless Gupta-Bleuler vacuum on the (1+1)-dimensional de Sitter space-time}.
\newblock {\em Physical Review D}, 57:6230--6241, May 1998.

\bibitem{gazeau2000gupta}
J.~Gazeau, J.~Renaud, and M.~Takook.
\newblock {Gupta-Bleuler quantization for minimally coupled scalar fields in de Sitter space}.
\newblock {\em Classical and Quantum Gravity}, 17(6):1415, 2000.

\bibitem{Gazeau:1984bs}
J.~P. Gazeau.
\newblock {Gauge Fixing and Gupta Bleuler Triplets in De Sitter {QED}}.
\newblock {\em Journal of Mathematical Physics}, 26:1847, 1985.

\bibitem{Gupta_1950}
S.~N. Gupta.
\newblock Theory of longitudinal photons in quantum electrodynamics.
\newblock {\em Proceedings of the Physical Society. Section A}, 63(7):681, 1950.

\bibitem{1494155620040901}
H.~Huang, H.~Chen, and P.~Zhang.
\newblock {Generalized Taft Algebras}.
\newblock {\em Algebra Colloquium}, 11(3):313--320, 2004.

\bibitem{kassel}
C.~Kassel.
\newblock {\em Quantum Groups}, volume 155 of {\em Graduate Texts in Mathematics}.
\newblock Springer New York, NY, 1995.

\bibitem{kerler2001non}
T.~Kerler and V.~V. Lyubashenko.
\newblock {\em Non-semisimple topological quantum field theories for 3-manifolds with corners}, volume 1765.
\newblock Springer Science \& Business Media, 2001.

\bibitem{kirillov1996inner}
A.~A. Kirillov~Jr.
\newblock On an inner product in modular tensor categories.
\newblock {\em Journal of the American Mathematical Society}, 9:1135--1169, 1996.

\bibitem{li2012greenringsgeneralizedtaft}
L.~Li and Y.~Zhang.
\newblock {The Green rings of the generalized Taft Hopf algebras}.
\newblock {\em Contemporary Mathematics}, 585:275--288, 2013.

\end{thebibliography}

\end{document}